\documentclass[11pt]{amsart}
\usepackage{graphicx}
\usepackage{amsmath,amsthm}
\usepackage{amssymb}
\usepackage{pinlabel}
\usepackage[left = 3cm, right = 3 cm , top = 3 cm , bottom = 3cm ]{geometry}
\usepackage[utf8]{inputenc}
\usepackage{todonotes}
\usepackage{enumerate}

\newtheorem{thm}{Theorem}[section]
\newtheorem{prop}[thm]{Proposition}
\newtheorem{lem}[thm]{Lemma}

\newtheorem{conj}[thm]{Conjecture}

\newtheorem{question}[thm]{Question}

\newtheorem{notation}[thm]{Notation}
\newcommand{\floor}[1]{\lfloor #1 \rfloor}

\theoremstyle{remark}
\newtheorem{remark}[thm]{Remark}

\newtheorem{ex}[thm]{Example}

\theoremstyle{definition}
\newtheorem{definition}[thm]{Definition}

\begin{document}
	
	\title[Non-negative integral matrices with given spectral radius]{Non-negative integral matrices with given spectral radius and controlled dimension}
	\author{Mehdi Yazdi}
%
%
%
	\maketitle
	
	\begin{abstract}
	A celebrated theorem of Douglas Lind states that a positive real number is equal to the spectral radius of some integral primitive matrix, if and only if, it is a Perron algebraic integer. Given a Perron number $p$, we prove that there is an integral irreducible matrix with spectral radius $p$, and with dimension bounded above in terms of the algebraic degree, the ratio of the first two largest Galois conjugates, and arithmetic information about the ring of integers of its number field. This arithmetic information can be taken to be either the discriminant or the minimal Hermite-like thickness. Equivalently, given a Perron number $p$, there is an irreducible shift of finite type with entropy $\log(p)$ defined as an edge shift on a graph whose number of vertices is bounded above in terms of the aforementioned data.
	\end{abstract}

\section{Introduction}

A real algebraic integer $\lambda \geq 1$ is \emph{Perron} if $\lambda$ is strictly larger than all its other Galois conjugates in absolute value. In what follows, all matrices are considered to be square size. For a real matrix $A$, by $A>0$ we mean that all entries of $A$ are positive. A non-negative real matrix $A$ is \emph{primitive} (or \emph{aperiodic}) if there is some natural number $k$ such that $A^k>0$. A non-negative real matrix is \emph{irreducible} if for any two indices $i$ and $j$, there is some natural number $k=k(i,j)$ such that $(A^k)_{ij}>0$. The Perron--Frobenius theorem implies that
\begin{enumerate}[I)]
	\item the spectral radius of any \emph{integral} primitive matrix is a Perron number; and
	\item the spectral radius of any \emph{integral} irreducible matrix is equal to the $n$th root of a Perron number where $n$ is the \emph{period} of the irreducible matrix. Moreover, the spectral radius of an integral non-negative non-nilpotent matrix is equal to the spectral radius of an integral irreducible one. 
\end{enumerate}

Lind \cite[Theorem 1]{lind1984entropies} proved a converse to the `integer' Perron--Frobenius theorem, namely for every Perron number $\lambda$ there is an integral primitive matrix with spectral radius equal to $\lambda$. It readily follows that for any Perron number $\lambda$ and any natural number $n$, there is an integral irreducible matrix with spectral radius equal to $\sqrt[n]{\lambda}$; see \cite[Theorem 3]{lind1984entropies}.

Associated to a non-negative and non-degenerate (i.e. with no zero rows or columns) integral matrix $A = [a_{ij}]$ with spectral radius $\lambda$ is a \emph{shift of finite type} with entropy equal to $\log(\lambda)$, which is defined as the \emph{edge shift} on a directed finite graph $G_A$ as follows. The graph $G_A$ has one vertex for each row of the matrix $A$, and there are exactly $a_{ij}$ oriented edges from the vertex $v_i$ to the vertex $v_j$. In particular, the dimension of the matrix $A$ (i.e. its number of rows or columns) is equal to the number of vertices of the graph $G_A$. The matrix is primitive if and only if the associated shift of finite type is topologically mixing. If the matrix is irreducible then the corresponding shift of finite type is called \emph{irreducible}. Irreducible shifts of finite type are those which are topologically transitive. See e.g. \cite{lind_marcus_2021}. 

Given a Perron number $\lambda$, the \emph{Perron--Frobenius degree} of $\lambda$, $d_{PF}(\lambda)$, is the least dimension of an integral \emph{primitive} matrix with spectral radius equal to $\lambda$. Clearly we have 
\[ d_{PF}(\lambda) \geq d, \]
where $d$ denotes the algebraic degree of $\lambda$ as an algebraic integer. It is easy to see that the equality happens if $\lambda$ is quadratic (see \cite[Remark 3.1]{yazdi2020lower}), so we consider the case of $d \geq 3$. Lind observed that if the \emph{trace} (i.e. sum of Galois conjugates) of $\lambda$ is negative, then $d_{PF}(\lambda)$ is strictly larger than $d$; see \cite[page 289]{lind1984entropies}. In \cite{yazdi2020lower}, using an idea of Lind, we gave a lower bound for $d_{PF}(\lambda)$ in terms of the layout of the two largest (in absolute value) Galois conjugates of $\lambda$ in the complex plane. As a corollary, it was shown that there are examples of cubic Perron numbers with arbitrarily large Perron--Frobenius degrees, a result previously known to Lind, McMullen, and Thurston, although unpublished. 

\begin{definition}
Given a Perron number $\lambda$, define the \emph{spectral ratio} of $\lambda$ as $\max_i \frac{|\lambda_i|}{\lambda}$, where $\lambda_i \neq \lambda$ are the remaining Galois conjugates of $\lambda$.
\end{definition}

\begin{notation}
For a Perron number $\lambda$, let $d_{PF}^{irr}(\lambda)$ be the smallest dimension of an integral \emph{irreducible} matrix with spectral radius $\lambda$. 
\end{notation}

In this paper, we give an explicit upper bound for $d_{PF}^{irr}(\lambda)$ in terms of the algebraic degree of $\lambda$, the spectral ratio of $\lambda$, and arithmetic information about the ring of integers $\mathcal{O}_\mathbb{K}$ of the number field $\mathbb{K}:= \mathbb{Q}(\lambda)$. This arithmetic quantity, which we call the \emph{minimal Hermite-like thickness} and denote it by $\tau_{\min}(\mathcal{O}_\mathbb{K})$, was previously defined by Bayer Fluckiger \cite{BAYERFLUCKIGER2006305} in relation to \emph{Minkowski's conjecture}; see Definition \ref{def:tau-min}. Intuitively, once an inner product is chosen on $\mathbb{R}^d$, the Hermite-like thickness is defined as the square of the covering radius, normalised properly, for the inclusion of the lattice $\mathcal{O}_\mathbb{K}$ in $\mathbb{R}^d$. The minimal Hermite-like thickness is then defined by taking the infimum of Hermite-like thickness over an appropriate space of inner products on $\mathbb{R}^d$. As a corollary of our main result, and using an inequality due to Banaszczyk and Bayer Fluckiger (see inequality (\ref{upper-bound-tau})), we obtain a similar bound in terms of the \emph{discriminant} $D_\mathbb{K}$ of $\mathcal{O}_{\mathbb{K}}$ instead of $\tau_{\min}(\mathcal{O}_\mathbb{K})$. See Definition \ref{def:discriminant}.

\begin{thm}
	\label{thm:upper-bound}
	Let $\lambda$ be a Perron number of algebraic degree $d \geq 3$ and spectral ratio $\rho$. Set $\mathbb{K} := \mathbb{Q}(\lambda)$. Let $\mathcal{O}_\mathbb{K}$ be the ring of integers of $\mathbb{K}$, and denote the discriminant and the minimal Hermite-like thickness of $\mathbb{K}$ by, respectively, $D_\mathbb{K}$ and $\tau_{\min}(\mathcal{O}_\mathbb{K})$. Then $d_{PF}^{irr}(\lambda)$ is bounded above by each of
	\[ \Big(\dfrac{8d}{1-\rho}\Big)^{d^2} \tau_{\min}(\mathcal{O}_\mathbb{K})^\frac{d}{2} \]
	and
	\[ \Big(\dfrac{8d}{1-\rho}\Big)^{d^2} \sqrt{D_\mathbb{K}}. \]
\end{thm}

\begin{remark}
Given a natural number $n$ and an integral irreducible matrix $A$ with spectral radius $\lambda$, one can readily construct an integral irreducible matrix $B$ with spectral radius $\sqrt[n]{\lambda}$ such that $\dim(B) = n \dim(A)$. Hence, Theorem \ref{thm:upper-bound} can be used to give an upper bound for $d_{PF}^{irr}(\sqrt[n]{\lambda})$. See e.g. the proof of \cite[Theorem 3]{lind1984entropies}.
\end{remark}

Theorem \ref{thm:upper-bound} immediately translates into the context of irreducible shifts of finite type with a given entropy. We can derive an upper bound for the Perron--Frobenius degree using Theorem \ref{thm:upper-bound}. See also Remark \ref{primitive-case} and Question \ref{ques:primitive-vs-irreducible}. First we need to introduce a notation.

\begin{notation}
	Let $\mathbb{K}$ be a real number field (i.e. with at least one real place), and $\rho \in (0,1)$. Set 
	\[ M = 1+\frac{4}{1-\rho}, \]
	and denote
	\[ \kappa(\mathbb{K},\rho):= \max \{ d_{PF}(\alpha) \hspace{1mm} | \hspace{1mm} \alpha \in [1,M]\cap \mathbb{K} \text{ is a Perron number}   \}. \]
	
	\label{notation:kappa}
\end{notation}

Note that for any $M>1$, there are only finitely many Perron numbers in the interval $[1,M]$ with degree at most $d$. Hence, $\kappa(\mathbb{K}, \rho)$ can be computed in theory.

\begin{thm}
	\label{thm:primitive}
	Let $\mathbb{\lambda}$ be a Perron number of degree $d \geq 3$ and spectral ratio $\rho$. Set $\mathbb{K}:= \mathbb{Q}(\lambda)$. Denote the bound from Theorem \ref{thm:upper-bound} by $B(\mathbb{K}, \rho)$; note that $d=[\mathbb{K}:\mathbb{Q}]$ is uniquely determined by $\mathbb{K}$. The Perron--Frobenius degree of $\lambda$ is bounded above by 
	\[ \max \{ 2^{d^2} B(\mathbb{K}, \rho), \kappa(\mathbb{K},\rho) \}. \]
\end{thm}

\subsection{Previous work} In \cite{lind1984entropies}, Lind gave a method for producing \emph{all} integral primitive matrices with a given Perron number $\lambda$ as their spectral radius. Let $B \colon \mathbb{R}^d \rightarrow \mathbb{R}^d$ be the \emph{companion matrix} associated to $\lambda$. In what follows, unless otherwise specified all eigenvectors are column (right) eigenvectors; similarly for eigenspaces. Pick an eigenvector $v$ for the eigenvalue $\lambda$, and denote the one dimensional eigenspace of $\lambda$ by $E_\lambda$. Let $E$ be the positive half-space associated to $v$, i.e. if $\pi_1 \colon \mathbb{R}^d \rightarrow E_\lambda$ is the projection along the complementary invariant subspace, then 
\[ E = \{ x \in \mathbb{R}^d \hspace{ 2mm} | \hspace{2mm} \pi_1(x) = r v \hspace{3mm} \text{for some } r>0 \}. \]

\begin{thm}[Lind]
Let $\lambda$ be a Perron number of algebraic degree $d$, and $B$ and $E$ be as above. If $A = [a_{ij}]$ is an $n$-dimensional primitive non-negative integral matrix with spectral radius $\lambda$, then there are $z_i \in \mathbb{Z}^d \cap E$ for $1 \leq i \leq n$ such that $Bz_j = \sum_{i=1}^{n} a_{ij}z_i$.  

Conversely, if $\lambda$, $B$, and $E$ are as above, and the points $z_i \in  \mathbb{Z}^d \cap E$ and a non-negative integral matrix $A=[a_{ij}]$ satisfy $Bz_j = \sum_{i=1}^{n} a_{ij}z_i$, then every irreducible component of $A$ has spectral radius equal to $\lambda$.
\end{thm}

The above theorem of Lind gives a practical way to produce an integral primitive matrix with a given spectral radius; see \cite[page 289]{lind1984entropies}. However, it does not tell us how to find such a matrix with smallest (or close to smallest) dimension, since we are not given control over the size of the coordinates of $z_i$. Nevertheless, the referee has kindly mentioned to me that there exists a simple, but not necessarily practical, algorithm that computes the Perron--Frobenius degree of a Perron number: Assume that $\lambda$ is given by its minimal polynomial, and denote the algebraic degree of $\lambda $ by $d$. For every positive integer $n$, there are only finitely many primitive $n \times n $ integral matrices with spectral radius less than or equal to $\lambda$. One can algorithmically enumerate these. For each of them, one can algorithmically determine whether the spectral radius is equal to $\lambda$. So for $n = d$, we can check whether there is an integral primitive matrix of dimension $n$ which has $\lambda$ as the spectral radius. Recursively, if we fail at $n$, then we try at $n+1$. By Lind's theorem we eventually find an $n$ where we succeed; that $n$ is equal to $d_{PF}(\lambda)$.

For matrices with non-negative integral \emph{polynomial} entries, the situation is different. See the work of Boyle and Lind, which gives a uniform upper bound (in fact a 2 by 2 matrix) in this context \cite{boyle2002small}. For the related topic of inverse spectral problem for non-negative integral matrices see the works of Boyle--Handelman \cite{boyle1991spectra} and Kim--Ormes--Roush \cite{kim2000spectra} and the references therein.

\subsection{Idea of the proof} In \cite{thurston2014entropy}, Thurston gave a simpler proof of Lind's converse to the integer Perron--Frobenius theorem. Our proof of Theorem \ref{thm:upper-bound} follows Thurston's approach, while controlling the dimension of a constructed matrix. 

The tensor product $\mathbb{Q}(\lambda) \otimes_{\mathbb{Q}} \mathbb{R}$ can be identified with $\mathbb{R}^d$. Let $M_\lambda$ be the linear endomorphism of $\mathbb{R}^d \cong \mathbb{Q}(\lambda) \otimes_{\mathbb{Q}} \mathbb{R}$ induced by multiplication by $\lambda$ in $\mathbb{Q}(\lambda)$. The eigenvalues of $M_\lambda$ are the Galois conjugates of $\lambda$. Then $\mathbb{R}^d$ decomposes into invariant subspaces of $M_\lambda$ corresponding to real places and pairs of conjugate complex places of $\lambda$; see the opening paragraphs to Section \ref{sec:proof}. Fix an eigenvector for $M_\lambda$ with eigenvalue $\lambda$, and consider the positive half-space corresponding to $\lambda$. 

We start with a polygonal cone with apex at the origin that lies in the positive half-space and is invariant under $M_\lambda$. We then perturb the vertices of the cone to obtain an invariant polygonal cone $\mathcal{C}$ with \emph{integral} vertices; see \textbf{Steps 2--4} of the proof of Theorem \ref{thm:upper-bound}. It is during this perturbation that the minimal Hermite-like thickness appears in the picture. Since the polygonal cone has integral vertices, the semigroup $S$ generated by the set of integral points in the cone $\mathcal{C}$ under addition of vectors is finitely generated; see Proposition \ref{semigroup}. The cardinality of a generating set for the semigroup $S$ gives an upper bound for the dimension of an integral non-negative matrix $A$ with spectral radius $\lambda$. Moreover, after possibly passing to an irreducible component of $A$, an integral irreducible matrix with spectral radius $\lambda$ is obtained; see \textbf{Step 5}. Finally we give an upper bound for the dimension of $A$; see \textbf{Step 6}. 

\subsection{Plan of the paper} In Section \ref{sec:prelim}, we present a few preliminary lemmas. The proof of Theorem \ref{thm:upper-bound} is given in Section \ref{sec:proof}. Theorem \ref{thm:primitive} is proved in Section \ref{sec:primitive}. In Section \ref{sec:questions}, a few questions are posed.

\subsection{Acknowledgement} I would like to thank Douglas Lind for sharing his intuition that some lattice property of the ring of integers should play a role in the Perron--Frobenius degree. Many thanks to Curtis T. McMullen and Eva Bayer Fluckiger for helpful comments connected to Remark \ref{rem:Minkowski}. I am grateful to the anonymous referee for his/her very helpful comments, in particular for explaining how the bound for primitive matrices in Theorem \ref{thm:primitive} can be obtained. During this work, the author was supported by a Glasstone Research Fellowship in Science, a Titchmarsh Fellowship, and a UKRI Postdoctoral Research Fellowship.

\section{Preliminaries}

\label{sec:prelim}

\subsection{Lattice points}

Throughout this article, by a \emph{lattice} $\Lambda \subset \mathbb{R}^d$ we mean a lattice of full rank; i.e. a discrete subgroup of $\mathbb{R}^d$ isomorphic to $\mathbb{Z}^d$. 

\begin{prop}
	Let $\Lambda \subset \mathbb{R}^d$ be a lattice, and $x_1 , \cdots, x_k \in \Lambda$. Denote by $\mathcal{C}$ the convex cone over the points $x_i$ inside $\mathbb{R}^d$, that is 
	\[ \mathcal{C}:= \{ \alpha_1 x_1 + \cdots + \alpha_k x_k \hspace{2mm} | \hspace{2mm}  \alpha_i \geq 0 \hspace{3mm} \text{for every }i \}. \]
	Let $S$ be the semigroup generated by elements of $\mathcal{C} \cap \Lambda$ under vector addition. Define the compact set $C$, and the finite set $C_{\Lambda} \subset C$ as 
	\[ C:=  \{ \alpha_1 x_1 + \cdots + \alpha_k x_k \hspace{2mm} | \hspace{2mm}  0 \leq \alpha_i  \leq 1 \hspace{3mm} \text{for every }i \},  \hspace{3mm} C_{\Lambda}:= C \cap \Lambda. \]
	Then $C_{\Lambda}$ is a finite generating set for the semigroup $S$.
	\label{semigroup}
\end{prop}

\begin{proof}
	For $\alpha \in \mathbb{R}$, denote the fractional part of $\alpha$ by $\{ \alpha \}$, and let $\floor{\alpha} = \alpha - \{ \alpha \}$. Any point $y \in S$ can be written as
	\[ y = \alpha_1 x_1 + \cdots + \alpha_k x_k = \big( \floor{\alpha_1} x_1 + \cdots + \floor{\alpha_k}x_k  \big) + \big( \{ \alpha_1 \} x_1 + \cdots + \{ \alpha_k \} x_k \big), \]
	where $\alpha_i , \floor{\alpha_i} \geq 0$ for each $i$. Note that we have $x_i \in C_{\Lambda}$ for each $i$, hence the first parenthesis is a sum of elements of $C_{\Lambda}$. As $y$ and the first parenthesis are both in $\Lambda$, the second parenthesis should represent a point in $\Lambda$ as well. On the other hand, the coefficients of the second parenthesis are in the interval $[0,1)$, and hence the second parenthesis lies in $C_\Lambda = C \cap \Lambda$. We have written $y$ as a sum of elements of $C_{\Lambda}$, hence $C_{\Lambda}$ is a generating set for the semigroup $S$. Since $C$ is compact and $\Lambda$ is a lattice, $C_{\Lambda}= C \cap \Lambda$ is a finite set. 
\end{proof}

By a \emph{Euclidean space} of dimension $d$ we mean a $d$-dimensional vector space $\mathbb{R}^d$ equipped with an inner product. A \emph{polytope} is the convex hull of finitely many points in $\mathbb{R}^d$. Let $\Lambda \subset \mathbb{R}^d$ be a lattice. If $\{ v_1, \cdots, v_d \} $ is a basis for $\Lambda \cong \mathbb{Z}^d$, then a \emph{fundamental domain} for $\Lambda$ is 
\[ \{ \alpha_1 v_1 + \cdots + \alpha_d v_d \hspace{2mm}|\hspace{2mm} 0 \leq \alpha_i < 1 \text{ for every }i\}.  \]
If $\mathbb{R}^d$ is a Euclidean space, then the \emph{covolume} of $\Lambda$ is defined as the volume of any fundamental domain for $\Lambda$. A \emph{lattice polytope} is a polytope whose vertices are lattice points.

\begin{prop}
	Let $\mathbb{R}^d$ be a Euclidean space, $\Lambda \subset \mathbb{R}^d$ be a lattice with covolume $\mathrm{Covol}(\Lambda)$, and $P \subset \mathbb{R}^d$ be a $d$-dimensional lattice polytope with volume $\mathrm{Vol}(P)$. Denote the number of lattice points in $P$ by $|P \cap \Lambda|$. Then
	\[ |P \cap \Lambda| \leq  \frac{\mathrm{Vol}(P)}{\mathrm{Covol}(\Lambda)} \cdot (d+1)!. \]
	The equality happens exactly when $P$ is a $d$-simplex with $|P \cap \Lambda|=d+1$.
	\label{number of lattice points}
\end{prop}

\begin{proof}
	First consider the special case that $P$ has exactly $d+1$ vertices, and every lattice point in $P$ is a vertex of $P$. If $v_0, v_1, \cdots, v_d \in \mathbb{R}^d$ are the vertices of $P$, then the volume of the parallelepiped formed by the vectors $v_1 - v_0 , \cdots , v_d - v_0$ is equal to $d! \times \mathrm{Vol}(P)$. Since the volume of this parallelepiped is at least as large as the volume of a fundamental domain for $\Lambda$, we have  
	\begin{align*}
	d! \times \mathrm{Vol}(P) \geq \mathrm{Covol}(\Lambda),
	\end{align*}
	implying that  
	\[ |P \cap \Lambda| = d+1  \leq \frac{\mathrm{Vol}(P)}{\mathrm{Covol}(\Lambda)} \cdot (d+1)!.  \]
	In general, decompose $P$ into $d$-simplices $\Delta_1, \cdots, \Delta_n$ with disjoint interiors such that each simplex $\Delta_i$ contains no lattice point except for its vertices. This can be done for example as follows. Decompose $P$ into smaller polyhedra by coning off from one of the vertices of $P$. Here by coning off from a point $v \in P$ we mean that for every facet $F$ of $P$, the polyhedron which is the convex hull of $F \cup v$ is added unless its dimension is strictly smaller than that of $P$; for example if the starting polyhedron is a polygon, then coning off from a vertex $v$ is just decomposing the polygon into triangles via adding all the diagonals emanating from $v$. For any of the resulting polyhedra, successively take a lattice point inside or on the boundary, and cone off from that lattice point. After finitely many repetitions, we arrive at the decomposition into $\Delta_i$. 
	
	The desired inequality follows from adding up the corresponding inequalities for simplices $\Delta_1, \cdots, \Delta_n$. Note that if $P$ is not a $d$-simplex or $|P \cap \Lambda| > d+1$, then at least one lattice point in $P \cap \Lambda$ is counted for more than one simplex $\Delta_i$, and so the inequality is strict.
\end{proof}

%
%

\subsection{Minkowski sum and difference}

For sets $A , B \subset \mathbb{R}^d$, define their \emph{Minkowski sum} as 
\[ A + B : = \{ a +b \hspace{2mm} | \hspace{2mm} a \in A, \text{ and } b \in B \} \subset \mathbb{R}^d.  \]
Intuitively, $A + B$ is the union of all translates of $A$ by elements of $B$
\[  A + B = \bigcup_{b \in B} (A + b). \]
Define the \emph{Minkowski difference} of $A$ and $B$ by 
\[ A \div B := \{ c \in \mathbb{R}^d \hspace{2mm} | \hspace{2mm} B+c \subseteq A \}.  \]
If $B$ is empty, $A \div B$ is, by convention, equal to $\mathbb{R}^d$. Intuitively, $A \div B$ is the intersection of all translates of $A$ by the antipodes of elements of $B$
\[ A \div B = \bigcap_{b \in B} (A - b). \]
We have used the rather odd notation $\div$ for the Minkowski difference, in order to distinguish it from the set
\[ \{ a - b \hspace{2mm} | \hspace{2mm} a \in A, \text{ and } b \in B\} \subset \mathbb{R}^d. \]
The Minkowski sum and difference are \emph{not} in general the inverse of each other.

\begin{lem}
The following properties hold for sets $A , B, C \subset \mathbb{R}^d$
\begin{align*}
&A \subset (A+B)\div B, \\
&A \subset B \implies A \div C \subset B \div C, \\
&A \subset B \implies A + C \subset B +C.
\end{align*}
Moreover, if $A$ and $B$ are non-empty compact, convex sets, then 
\[ (A+B) \div B = A. \]
\label{minkowski-difference}
\end{lem}

\begin{proof}
The first three properties directly follow from the definition. We sketch the proof of the last implication, and refer the reader to e.g. \cite[Lemma 3.1.11, and Section 1.7]{schneider_2013} for details. Pick an inner product $\langle \cdot , \cdot \rangle $ on $\mathbb{R}^d$, and denote the \emph{support functions} for $A$ and $B$ by, respectively, $h_A$ and $h_B$. By definition, for any $x \in \mathbb{R}^d$ 
\[ h_A(x): = \sup\{ \langle x, y \rangle  \hspace{2mm} | \hspace{2mm} y \in A \}, \]
$h_B$ is defined similarly. By the \emph{separation theorem} for convex sets, for a non-empty compact convex set $A $ we have 
\[ a \in A \iff \langle a , x \rangle  \leq h_A(x) \hspace{2mm} \text{for all } x \in \mathbb{R}^d.  \]
Assuming $x \in (A+B)\div B$, we would like to show that $x \in A$. By hypothesis, $x + B \subset A + B$. Equivalently, the support function for $x+B$ does not exceed that of $A+B$ pointwise. This implies 
\[ h_{\{x\}} + h_B \leq h_A +h_B, \]
using the fact that $h_{A+B} = h_A + h_B$ for non-empty compact convex sets $A$ and $B$; see \cite[Theorem 1.7.5]{schneider_2013}. Cancelling $h_B$ from both sides gives us the inequality $h_{\{x\}} \leq h_A$, implying that $x \in A$. 
\end{proof}

\section{Proof of Theorem \ref{thm:upper-bound}}
\label{sec:proof}
We follow Bayer Fluckiger \cite{BAYERFLUCKIGER2006305} and Jarvis \cite{jarvis2014algebraic} for the definitions below. Let $\lambda$ be an algebraic integer, and $\mathbb{K} = \mathbb{Q}(\lambda)$ be the number field obtained by adjoining $\lambda$ to $\mathbb{Q}$. We may interpret the points of $\mathbb{K}$ as lying in a $d$-dimensional real linear space as follows. Let $\sigma_1 , \cdots, \sigma_r$ be the real embeddings of $\mathbb{K}$, and $\sigma_{r+1}, \overline{\sigma}_{r+1}, \cdots, \sigma_{r+s}, \overline{\sigma}_{r+s}$ be the pairwise conjugate complex embeddings of $\mathbb{K}$, where 
\begin{eqnarray}
r+2s = d .
\end{eqnarray}
Consider the embedding 
\begin{align*}
&\sigma \colon \mathbb{K} \longrightarrow \mathbb{R}^r \times \mathbb{C}^s, \\
& \sigma(x) = (\sigma_1(x), \cdots, \sigma_{r+s}(x)). 
\end{align*}
We may identify $\mathbb{C}$ with $\mathbb{R}^2$, by identifying $a+bi$ with $(a,b)$, then $\sigma$ becomes an embedding 
\[ \sigma \colon \mathbb{K} \longrightarrow \mathbb{R}^d. \]
The mapping $\sigma \colon \mathbb{K} \rightarrow \mathbb{R}^d$ identifies the vector space $\mathbb{R}^d$ with the tensor product $\mathbb{K}_\mathbb{R} : = \mathbb{K} \otimes_{\mathbb{Q}} \mathbb{R}$, 
\begin{align*}
&\mathbb{K} \otimes_{\mathbb{Q}} \mathbb{R} \longleftrightarrow \mathbb{R}^d \\
&x \otimes a \mapsto (\sigma x)a.
\end{align*} 

$\mathbb{R}^r \times \mathbb{C}^s$ has a canonical involution, which is identity on $\mathbb{R}^r$ and complex conjugation on $\mathbb{C}^s$. Let 
\begin{eqnarray}
\mathfrak{B}:= \{ \alpha \in \mathbb{R}^r \times \mathbb{C}^s \hspace{2mm} | \hspace{2mm} \alpha = \bar{\alpha}, \text{ and all components of }\alpha \text{ are positive}  \}.
\label{def:space-of-inner-products}
\end{eqnarray}

Given $\alpha \in \mathfrak{B}$, define the symmetric positive definite bilinear form $q_\alpha$ on $\mathbb{K}_\mathbb{R}$ by 
\begin{align*}
&q_\alpha \colon \mathbb{K}_\mathbb{R} \times \mathbb{K}_\mathbb{R} \longrightarrow \mathbb{R}\\
&(x,y) \mapsto \text{Trace}(\alpha x \bar{y}).
\end{align*}
Here $\text{Trace}(x_1 , \cdots, x_{r+s}) := \sum_{i=1}^{r} x_i + \sum_{j=r+1}^{r+s} (x_j + \overline{x}_j)$ denotes the trace of $\mathbb{K}_\mathbb{R} \cong \mathbb{R}^r \times \mathbb{C}^s$, and $\alpha x \bar{y}$ denotes component wise product in $\mathbb{R}^r \times \mathbb{C}^s$. Then $q_\alpha$ induces the norm $|\cdot |_\alpha$ on $\mathbb{K}_\mathbb{R}$ given by the following formula, where $\alpha = (\alpha_1, \cdots, \alpha_{r+s})$ with $\alpha_i \in \mathbb{R}^{>0}$
\begin{eqnarray}
|x|^2_\alpha = \sum_{i=1}^{r}\alpha_i \cdot |\sigma_i(x)|^2 + 2 \sum_{j=r+1}^{r+s}\alpha_j \cdot |\sigma_j(x)|^2.
\end{eqnarray}

Denote the ring of integers of $\mathbb{K}$ by $\mathcal{O}_\mathbb{K}$. Let $(\mathcal{O}_\mathbb{K}, q_\alpha)$ denote the lattice $\mathcal{O}_\mathbb{K}$ equipped with the inner product $q_\alpha$. The \emph{maximum} of $(\mathcal{O}_\mathbb{K}, q_\alpha)$ is defined as 

\begin{eqnarray*}
	\max(\mathcal{O}_\mathbb{K}, q_\alpha) = \inf \{ u \in \mathbb{R} \hspace{1mm} | \hspace{1mm} \text{for all } x \in \mathbb{K}_\mathbb{R}, \text{ there exists } y \in \mathcal{O}_\mathbb{K} \text{ with } q_\alpha(x-y, x-y) \leq u \}.
\end{eqnarray*}

\vspace{2mm}

The \emph{covering radius} of $(\mathcal{O}_\mathbb{K}, q_\alpha)$ is, by definition, the square root of $\max(\mathcal{O}_\mathbb{K}, q_\alpha)$. Define the \emph{determinant} of $(\mathcal{O}_\mathbb{K}, q_\alpha)$ as the determinant of the matrix of $q_\alpha$ in a $\mathbb{Z}$-basis of $\mathcal{O}_\mathbb{K}$; i.e. if $\omega_1, \cdots, \omega_d$ is a basis for the abelian group $\mathcal{O}_\mathbb{K} \cong \mathbb{Z}^d$ (under addition) then $\det(\mathcal{O}_\mathbb{K},q_\alpha)$ is the determinant of the $d \times d$ matrix $(q_\alpha (w_i, w_j))$. With this definition, the determinant of $(\mathcal{O}_\mathbb{K}, q_\alpha)$ is equal to the \emph{square} of the volume of any fundamental domain for the lattice $(\mathcal{O}_\mathbb{K},q_\alpha)$. We remark that some texts define the determinant as the volume of a fundamental domain, but we preferred to follow Bayer Fluckiger's convention as in \cite{BAYERFLUCKIGER2006305}.

\begin{definition}
Define the \emph{Hermite-like thickness} $\tau(\mathcal{O}_\mathbb{K}, q_\alpha)$ of $(\mathcal{O}_\mathbb{K}, q_\alpha)$ as 
\begin{align*}
\tau(\mathcal{O}_\mathbb{K}, q_\alpha): = \frac{\max(\mathcal{O}_\mathbb{K}, q_\alpha) }{\det (\mathcal{O}_\mathbb{K}, q_\alpha)^\frac{1}{d}}.
\label{def:Hermite-like-thickness}
\end{align*}
Define the \emph{minimal Hermite-like thickness} as 
\begin{align*}
\tau_{\min} (\mathcal{O}_\mathbb{K}) = \inf \{  \tau(\mathcal{O}_\mathbb{K}, q_\alpha) \hspace{2mm} | \hspace{2mm} \alpha \in \mathfrak{B})  \},
\end{align*}
where $\mathfrak{B}$ is as in (\ref{def:space-of-inner-products}). 
\label{def:tau-min}
\end{definition}

\begin{remark}
Although we called $\tau_{\min}(\mathcal{O}_{\mathbb{K}})$ the minimal Hermite-like thickness, it should be noted that the minimum is taken over the set of inner products coming from elements of $\mathfrak{B}$ and not all possible inner products on $\mathbb{R}^d$.
\end{remark}

It is clear that the concepts of maximum, covering radius, and Hermite-like thickness can be defined more generally for a lattice in a Euclidean space; see \cite{BAYERFLUCKIGER2006305}. 

\begin{definition}
Let $\mathbb{K}$ be a number field. Assume that $\omega_1, \cdots, \omega_d$ is any integral basis for $\mathcal{O}_\mathbb{K}$. Denote the complete list of places of $\mathbb{K}$ by $\sigma_1, \cdots, \sigma_d$. The \emph{discriminant} of $\mathbb{K}$ is defined as the square of the determinant of the $d \times d$ matrix $(\sigma_i(\omega_j))$.
\label{def:discriminant}
\end{definition} 

See \cite[Chapters 3 and 7]{jarvis2014algebraic} for further properties of the discriminant as well as the embedding of $\mathcal{O}_\mathbb{K}$ in $\mathbb{R}^d$. 

\newtheorem*{thm:upper-bound}{Theorem \ref{thm:upper-bound}}
\begin{thm:upper-bound}
Let $\lambda$ be a Perron number of algebraic degree $d \geq 3$ and spectral ratio $\rho$. Set $\mathbb{K} := \mathbb{Q}(\lambda)$. Let $\mathcal{O}_\mathbb{K}$ be the ring of integers of $\mathbb{K}$, and denote the discriminant and the minimal Hermite-like thickness of $\mathbb{K}$ by, respectively, $D_\mathbb{K}$ and $\tau_{\min}(\mathcal{O}_\mathbb{K})$. Then $d_{PF}^{irr}(\lambda)$ is bounded above by each of
\[ \Big(\dfrac{8d}{1-\rho}\Big)^{d^2} \tau_{\min}(\mathcal{O}_\mathbb{K})^\frac{d}{2} \]
and
\[ \Big(\dfrac{8d}{1-\rho}\Big)^{d^2} \sqrt{D_\mathbb{K}}. \]
\end{thm:upper-bound}

\begin{proof}
Let $\mathbb{K}_\mathbb{R}:= \mathbb{K} \otimes_\mathbb{Q} \mathbb{R}$. Let $\sigma_1 , \cdots, \sigma_{r+s}$ be as before, and identify $\mathbb{K}_\mathbb{R}$ with $\mathbb{R}^r \times \mathbb{C}^s \cong \mathbb{R}^d$. A place $\sigma_j$ is a field homomorphism $\sigma_j \colon \mathbb{Q}(\lambda) \rightarrow (\mathbb{R} \text{ or } \mathbb{C})$ and hence $\sigma_j$ is completely determined by $\sigma_j(\lambda)$ which is one of the Galois conjugates of $\lambda$. Assume that $\sigma_1$ is the real place corresponding to $\lambda$ itself; i.e. $\sigma_1(\lambda)=\lambda$. Let $M_\lambda$ be the linear endomorphisms of $\mathbb{K}_\mathbb{R}$ induced by multiplication by $\lambda $ in $\mathbb{K}$. The eigenvalues of $M_\lambda$ are the Galois conjugates of $\lambda$. For any Galois conjugate $\lambda_i$, denote by $E_i$ the invariant subspace for $M_\lambda$ with eigenvalue $\lambda_i$, and let $\pi_i \colon \mathbb{K}_\mathbb{R} \rightarrow E_i$ be the projection along the complementary invariant subspace. Therefore, $\pi_i$ is the projection onto the $i$th factor under the identification $\mathbb{K}_\mathbb{R} \cong \mathbb{R}^r \times \mathbb{C}^s$.

Before going into the details, we explain the main idea when $\lambda$ is a cubic algebraic integer that is not totally real. In order to find a non-negative integral matrix with spectral radius $\lambda$, we follow Thurston's proof of Lind's theorem. See \cite[pages 353--354]{thurston2014entropy} and \cite{lind1984entropies}. Let $E_1$ be the one dimensional invariant subspace for $M_\lambda$ with eigenvalue $\lambda$, and $E_2$ be the two dimensional invariant subspace corresponding to the pair of complex Galois conjugates $\{ \delta, \overline{\delta} \} $ of $\lambda$. The endomorphism $M_\lambda$ of $\mathbb{K}_\mathbb{R}$ leaves $E_1$ and $E_2$ invariant, and it acts on $E_1 \cong \mathbb{R}$ and $E_2 \cong \mathbb{C}$ by multiplication by, respectively, the numbers $\lambda$ and $\delta$. Pick a large positive integer $N = N(\delta, \lambda)$ such that if $P_\delta$ is a regular $N$-gon inscribed in a circle of radius $R$ around the origin in $E_2$, then $P_\delta \subset E_2 \cong \mathbb{C}$ is invariant under multiplication by the complex number $\delta/\lambda$. Such an integer $N$ exists since by the Perron condition the absolute value of $\delta/\lambda$ is strictly less than $1$. Let $v$ be an eigenvector of $M_\lambda$ with eigenvalue $\lambda$, and $E_{2}^v$ be the affine plane $Rv + E_2$. Then $E_2^v$ lies in the positive half-space $E$ corresponding to $\lambda$ and $v$. 

Denote the vertices of the shifted polygon $P_v := Rv+ P_\delta \subset E_2^v$ by $v_1 , \cdots, v_k$, and choose integral points $z_1, \cdots, z_k \subset \mathbb{R}^3$ such that the distance between $z_i$ and $v_i$ is `small'. Since the cone over the points $v_1, \cdots, v_k$ lies in the positive half-space and is invariant under $M_\lambda$, and the distance between $z_i$ and $v_i$ is small, it is reasonable to expect that the cone $\mathcal{C}$ over the points $z_i$ also lies in the positive half-space $E$ and is invariant under $M_\lambda$ for `large' $R$. Let $S$ be the semigroup generated by the set of integral points in the cone $\mathcal{C}$ under vector addition. Then $M_\lambda$ preserves $S$ and induces an action $M_\lambda^S$ on $S$. Moreover, $S$ has a finite generating set, and we can estimate an upper bound for the size $|G|$ of a generating set $G$ using Proposition \ref{semigroup}. If we write the action of $M_\lambda^S$ on $S$ in the generating set $G$, we obtain a non-negative integral matrix of size $|G|$ whose spectral radius is equal to $\lambda$. The details of the proof are as follows.\\

\textbf{Step 1: Choosing an  inner product on $\mathbb{R}^r \times \mathbb{C}^s$.} \\

Define $\mathfrak{B}$ as in (\ref{def:space-of-inner-products}). Pick $\alpha \in \mathfrak{B}$ and equip $\mathbb{R}^r \times \mathbb{C}^s$ with the inner product $q_\alpha$. Note that for the norm $|\cdot|_\alpha$ and for any $x \in \mathbb{K}_\mathbb{R}$ 
\begin{eqnarray}
|x|_\alpha^2 = \sum_{j=1}^{r+s} |\pi_j(x)|_\alpha^2, 
\label{norm-sum}
\end{eqnarray}
and hence
\begin{eqnarray}
|x|_\alpha \geq |\pi_j(x)|_\alpha \hspace{6mm} \text{for} \hspace{2mm} 1 \leq j \leq r+s.
\label{norm-comparison}
\end{eqnarray}

Let $\ell$ be the covering radius of $(\mathcal{O}_\mathbb{K}, q_\alpha)$. Then
\begin{eqnarray}
\ell : = \max(\mathcal{O}_\mathbb{K}, q_\alpha)^\frac{1}{2} = \tau(\mathcal{O}_\mathbb{K}, q_\alpha)^\frac{1}{2} \cdot \det (\mathcal{O}_\mathbb{K}, q_\alpha)^\frac{1}{2d}.
\end{eqnarray} 

\vspace{1.5mm}
\textbf{Step 2: Defining the polygon $P_j$ in the invariant subspace $E_j$ for each $j>1$.}\\

Define 
\begin{eqnarray}
\rho_j = \frac{|\sigma_j(\lambda)|}{\lambda} \in \mathbb{R} \hspace{6mm} \text{for} \hspace{2mm}  1 < j \leq r+s.
\end{eqnarray}
Hence 
\begin{eqnarray}
\rho = \max_{j>1} \{ \rho_j \}. 
\end{eqnarray}
By the Perron condition, $\rho_j \in (0,1)$ for each $j>1$. Set 
\begin{eqnarray}
R_j = \frac{(2 \sqrt{d}+4)\ell}{1 - \rho} \hspace{6mm} \text{for} \hspace{2mm}  1 < j \leq r.
\label{R-real}
\end{eqnarray} 

\begin{remark}
Clearly $R_j$ does not depend on $1<j \leq r$. However, we decided that this notation would be more suitable if one would like to improve the bounds in the article by substituting $\rho_j$ instead of $\rho$ in the definition of $R_j$. Similarly, in what follows $R_j$ for $j>r$ will not depend on $j$.
\end{remark}

For each real place $\sigma_j$ with $j>1$, define $P_j$ as the set of points of distance at most $R_j$ from the origin in $E_j$; in particular $P_j$ is an interval. Note, for future use, that 
\begin{eqnarray}
R_j \leq \Big( \frac{8 \sqrt{d}}{1 - \rho} \Big) \ell  \hspace{6mm} \text{for} \hspace{2mm}  1 < j \leq r.
\label{upper-bound-R-real}
\end{eqnarray}
For $j>r$, define the natural number $N_j \geq 3$ as the smallest positive integer satisfying
\begin{eqnarray}
N_j^2 \geq \frac{2\sqrt{d}+9}{1- \rho}.
\label{lower-bound-N}
\end{eqnarray}
Note, for later use, that
\begin{eqnarray}
N_j^2 \leq  \frac{16 \sqrt{d}}{1-\rho}. 
\label{upper-bound-R-complex}
\end{eqnarray}
In the above, we used the fact that the smallest whole square exceeding $u \geq 16$ is less than or equal to $(\sqrt{u}+1)^2 \leq 3u/2 +1$. Set 
\begin{eqnarray}
R_j = N_j^2 \ell \hspace{6mm} \text{for} \hspace{2mm}  r < j \leq r+s.
\label{R-complex}
\end{eqnarray}
For $j>r$, define the solid polygon $P_j \subset E_j$ as a regular $N_j$-gon inscribed in the circle of radius $R_j$ around the origin. Define the polytope $P$ as the product of $P_j$ for $j>1$
\[ P :=  \{ z \in \prod_{j >1} E_j \hspace{2mm}| \hspace{2mm}\pi_j(z) \in P_j \hspace{2mm} \text{for every }j \} \subset \prod_{j>1} E_j.  \]
The number of vertices of $P$ is equal to the product of the number of vertices of $P_j$ for $j>1$, which is equal to $2^{r-1} \times \prod_{j >r} N_j$. \\

\textbf{Step 3: Defining the cone $\mathcal{C}$ in the positive half-space of $\lambda$.}\\

Let $v \in E_1$ be the positive (with respect to $\pi_1$) unit length eigenvector for the eigenvalue $\lambda$. Here unit length is considered with respect to the distance $| \cdot |_\alpha$. Let $E$ be the positive half-space corresponding to $v$. Set 
\begin{eqnarray}
L = \max\{ R_j\hspace{2mm} | \hspace{2mm} 1<j \leq r+s \}.
\label{L}
\end{eqnarray} 
Note, for future use, that 
\begin{eqnarray}
L > 3 \ell.
\label{L-ell}
\end{eqnarray}
Denote by $P_v$ the translated polytope $Lv + P$, and let $\mathcal{C}_v \subset E$ be the cone over the polytope $P_v$. Equivalently, if $\frac{1}{L} \cdot P_j$ is the dilation of $P_j$ by the factor $\frac{1}{L}$ centered at the origin of $E_j$, then
\begin{eqnarray}
 \mathcal{C}_v =  \{ z \in E \hspace{2mm} | \hspace{2mm} \frac{\pi_j(z)}{|\pi_1(z)|_\alpha} \in \frac{1}{L} \cdot P_j \hspace{3mm} \text{for every } j>1 \}.
 \label{cone-Cv}
\end{eqnarray} 
Denote the vertices of $P_v$ by $v_1, v_2, \cdots, v_k$, where 
\begin{eqnarray}
k = 2^{r-1} \times \prod_{j >r} N_j .
\label{number-of-vertices-k}
\end{eqnarray} 
Pick integral points $z_1, z_2, \cdots , z_k$ in $\mathcal{O}_\mathbb{K}$ such that the distance between $v_i$ and $z_i$ does not exceed the covering radius $\ell$ of $(\mathcal{O}_\mathbb{K}, q_\alpha)$. 

First we show that each $z_i$ lies in the positive half-space $E$. Identify the subspace $E_1$ with $\mathbb{R}$ via
\[ x = rv \in E_1 \longleftrightarrow r  \in \mathbb{R}. \]
It is enough to show that for each $i$ we have $\pi_1(z_i) >0$. By the triangle inequality
\[ \pi_1(z_i) \geq \pi_1( v_i) - |\pi_1(v_i - z_i)| \geq  \pi_1( v_i) - |v_i - z_i|_\alpha \geq  L - \ell > 0,  \]
where the second inequality holds by (\ref{norm-comparison}), and the last inequality is true by (\ref{L-ell}).

Define $\mathcal{C}$ as the cone over the points $z_1, z_2, \cdots, z_k$ with apex at the origin
\begin{eqnarray}
\mathcal{C} = \{ z \in E \hspace{2mm} | \hspace{2mm} z = \beta_1 z_1 + \cdots + \beta_k z_k, \hspace{3mm} \beta_i \geq 0 \hspace{3mm} \text{for every }i \}.
\end{eqnarray}
For each $z_i$, define $w_i$ as the intersection point of the ray through $z_i$ from the origin and the affine hyperplane $Lv+E_1^c$, where $E_1^c = \oplus_{j>1} E_j$ is the complementary invariant subspace to $E_1$. Then 
\begin{eqnarray}
 \mathcal{C} =  \{ z \in E \hspace{2mm} | \hspace{2mm}   z = \beta_1 w_1 + \cdots + \beta_k w_k, \hspace{3mm} \beta_i \geq 0 \hspace{3mm} \text{for every }i  \}.
\end{eqnarray}
Equivalently, if $Q_v \subset Lv + E_1^c$ is the convex hull of $w_1, w_2, \cdots, w_k$, then
\begin{eqnarray}
\mathcal{C} = \{ z \in E \hspace{2mm} | \hspace{2mm} \frac{z}{|\pi_1(z)|_\alpha} \in \frac{1}{L} \cdot Q_v \} = \{ z \in E \hspace{2mm} | \hspace{2mm} \frac{Lz}{|\pi_1(z)|_\alpha} \in Q_v \}.
\label{cone-C}
\end{eqnarray}

Hence, $\mathcal{C}$ is the cone over the points $z_1, z_2, \cdots, z_k$, or equivalently the cone over the points $w_1, w_2, \cdots, w_k$. In what follows, we will use the two descriptions of $\mathcal{C}$ as needed.\\

\textbf{Step 4: Showing that the cone $\mathcal{C}$ is invariant.}\\ 

We need a few lemmas in order to prove that $\mathcal{C}$ is invariant. 

\begin{lem}
	For each $i$ we have
	\begin{align*}
	&|v_i|_\alpha \leq \sqrt{d}L \\
	&|v_i - w_i|_\alpha \leq (2\sqrt{d}+2)\ell.
	\end{align*}
	\label{upper-bound-v}
\end{lem}

\begin{proof}
	For the first inequality, we have
	\begin{align*}
	|v_i|_\alpha^2 &= \sum_{j=1}^{r+s} |\pi_j(v_i)|_\alpha^2 \leq L^2 + \sum_{j>1} R_j^2 \leq (r+s) L^2 \leq d L^2,
	\end{align*}	
	where we used (\ref{norm-sum}) and (\ref{L}). 
	
	Assume that $w_i = r_i z_i$ for some positive real number $r_i$. Note that 
	\begin{align*}
	&|\pi_1(z_i) - \pi_1(v_i)|_\alpha=  |\pi_1(z_i -v_i)|_\alpha \leq |z_i - v_i|_\alpha \leq \ell  \\
	& \implies \pi_1(z_i) = \pi_1(v_i)+ \pi_1(z_i -v_i) \in [L - \ell , L + \ell],
	\end{align*}
	where we used (\ref{norm-comparison}) in the first line above. Therefore
	\begin{align*}
	r_i = \frac{\pi_1(w_i)}{\pi_1(z_i)} = \frac{L}{\pi_1(z_i)} \in [\frac{L}{L+\ell}, \frac{L}{L - \ell}]  \\ 
	\implies r_i \leq \frac{L}{L - \ell}, \hspace{3mm} \text{and} \hspace{3mm} |r_i -1| \leq \frac{\ell }{L - \ell} .
	\end{align*}
	Hence, by the triangle inequality  
	
	\begin{align*}
	&|w_i - v_i |_\alpha = |r_iz_i - v_i|_\alpha \leq |r_iz_i - r_i v_i|_\alpha+|r_iv_i - v_i|_\alpha  \\ 
	&= r_i|z_i - v_i|_\alpha+|r_i-1||v_i|_\alpha  \leq (\frac{L}{L-\ell})\ell + (\frac{\ell}{L-\ell})\sqrt{d}L \\ 
	&= (\sqrt{d}+1)\frac{L \ell}{L- \ell} \leq (2\sqrt{d}+2)\ell.
	\end{align*}
	The last implication above, $L \leq 2(L- \ell)$, holds by (\ref{L-ell}).
\end{proof}

Set 
\begin{eqnarray}
b = (2 \sqrt{d}+2) \ell.
\end{eqnarray}

\begin{lem}
The following estimates hold for every $z_i$, where $M_\lambda$ is the linear endomorphism of $\mathbb{K}_\mathbb{R}$ induced by multiplication by $\lambda$ in $\mathbb{K}$.

\begin{enumerate}[a)]
	\item \[ \frac{|\pi_j (M_\lambda(z_i))|_\alpha}{|\pi_1 (M_\lambda(z_i))|_\alpha}  \leq \rho_j \Big( \frac{R_j + \ell}{L - \ell} \Big) \hspace{6mm} \text{for every } j>1.\]
	\item \[ \frac{|\pi_j (M_\lambda(z_i))|_\alpha}{|\pi_1 (M_\lambda(z_i))|_\alpha} \leq \frac{R_j - b}{L} \hspace{6mm} \text{for every } 1<j\leq r.  \]
	\item \[  \frac{|\pi_j(M_\lambda(z_i))|_\alpha}{|\pi_1(M_\lambda(z_i))|_\alpha} \leq \frac{1}{L} \bigg( R_j \big(1 - \frac{5}{N_j^2} \big) - b \bigg) \hspace{6mm} \text{for every } j>r. \]
\end{enumerate}
\label{estimates}
\end{lem}

\begin{proof}
\begin{enumerate}[a)]
	\item We have 
	\begin{align*}
	\frac{|\pi_j (M_\lambda(z_i))|_\alpha}{|\pi_1 (M_\lambda(z_i))|_\alpha} &= \frac{|\sigma_j(\lambda)|}{|\sigma_1(\lambda)|} \cdot \frac{|\pi_j(z_i)|_\alpha}{|\pi_1(z_i)|_\alpha} = \rho_j \cdot \frac{|\pi_j(z_i)|_\alpha}{|\pi_1(z_i)|_\alpha} \\ 
	&\leq \rho_j \cdot  \frac{|\pi_j(v_i)|_\alpha + |\pi_j(z_i-v_i)|_\alpha}{|\pi_1(v_i)|_\alpha - |\pi_1(v_i-z_i)|_\alpha} \nonumber\\
	& \leq \rho_j \cdot  \frac{|\pi_j(v_i)|_\alpha + |z_i-v_i|_\alpha}{|\pi_1(v_i)|_\alpha - |v_i-z_i|_\alpha} \leq \rho_j \Big( \frac{R_j + \ell}{L - \ell} \Big),
	\label{first-bound}
	\end{align*}
	where we have used inequality (\ref{norm-comparison}).
	
	\item Assume that $\sigma_j$ is a real place where $1 < j \leq r$, and set $R'_j = R_j/\ell$, $L' = L/ \ell$, and $b' = b / \ell$. Then
	\begin{align*}
	 \rho_j \Big( \frac{R_j + \ell}{L - \ell} \Big) \leq \frac{R_j - b}{L} &\iff \rho_j \Big( \frac{R'_j+1}{L'-1} \Big) \leq \frac{R'_j - b'}{L'}  \\ 
	 &\iff (\rho_j+b') L'+ R_j'- b' \leq (1-\rho_j)L'R'_j. 
	 \end{align*}
	 Hence, using $\rho_j \in (0,1)$, it is enough to have
	 \begin{align*}
	 (1+b')L'+ R_j' \leq (1-\rho_j)L'R'_j &\iff \frac{1+b'}{R'_j}+\frac{1}{L'} \leq 1 - \rho_j \\ 
	 & \iff \frac{2\sqrt{d}+3}{R'_j}+\frac{1}{L'} \leq 1 - \rho_j.
	 \end{align*}
	 On the other hand, $L' \geq R_j'$ by (\ref{L}). Hence, it suffices to have  
	 \begin{align*}
	 \frac{2\sqrt{d}+4}{R'_j} \leq 1 - \rho_j \iff R'_j \geq \frac{2\sqrt{d}+4}{1 - \rho_j},
	 \end{align*} 
	 which holds by (\ref{R-real}).
	\item Now assume that $\sigma_j$ is a complex place where $j>r$. Then
	\begin{align*}
	\rho_j \Big( \frac{R_j + \ell}{L - \ell} \Big) \leq \frac{1}{L} \bigg( R_j \big(1 - \frac{5}{N_j^2} \big) - b \bigg) & \iff \rho_j (R_j +\ell) \leq \Big( \frac{L - \ell}{L} \Big) \Big( R_j \big(1 - \frac{5}{N_j^2} \big) - b \Big).
	\end{align*}
	By (\ref{L}), we have $L \geq R_j = N_j^2 \ell$, and so
	\begin{align*}
	\frac{L-\ell }{L} = 1 - \frac{\ell}{L} \geq 1 - \frac{1}{N_j^2}.
	\end{align*}
	Set $b' = b/\ell$. After substituting $R_j = N_j^2 \ell$ and dividing both sides by $\ell$, it is enough to have 
	\begin{align*}
	& \rho_j (N_j^2+1) \leq \Big( 1- \frac{1}{N_j^2} \Big)\Big( N_j^2-5-b' \Big).
	\end{align*}
	Multiplying both sides by $N_j^2$ and then expanding, the last inequality is equivalent to
	\begin{align*}
	N_j^4(1- \rho_j) + (5+b') \geq N_j^2(b'+ 6 +\rho_j).
	\end{align*}
	So, using $\rho_j \in (0,1)$ and neglecting the positive term $5+b'$, it suffices to have
	\begin{align*}
	N_j^4(1- \rho_j) \geq N_j^2(b'+ 7) \iff N_j^2 \geq \frac{b'+7}{1 - \rho_j} = \frac{2\sqrt{d}+9}{1- \rho_j},
	\end{align*}
	and the last inequality holds by (\ref{lower-bound-N}).
\end{enumerate}
\end{proof}

Recall that $\mathcal{C}$ is the cone over the points $w_1, \cdots, w_k$ with apex at the origin, i.e. the cone over the polytope $Q_v$ (i.e. the convex hull of $w_i$). For each $i$, the point $w_i$ of $Q_v$ is of distance at most $b = (2\sqrt{d}+2)\ell$ from the corresponding point $v_i$ of $P_v$. Since $b$ is `small', we expect the polytope $Q_v$ to contain a `smaller version' of $P_v$. More precisely, define
\begin{align*}
P_v^b = \{ x \in P_v \hspace{2mm} | \hspace{2mm} \text{dist}(x, \partial P_v) \geq b \},
\end{align*}
where 
\begin{align*}
\text{dist}(X, Y) : = \min \{ |x - y|_\alpha \hspace{2mm} | \hspace{2mm} x \in X, \hspace{2mm}  y \in Y \}.
\end{align*}

\begin{lem}
We have $P_v^b \subset Q_v$.
\label{inclusion-U}
\end{lem}

\begin{proof}
Let $B$ be the ball of radius $b$ around the origin in $E_1^c = \oplus_{j>1} E_j$. We show that 
\[ P_v \subset Q_v + B,  \]
where $Q_v + B$ denotes the Minkowski sum. To see this, pick an arbitrary point $x \in P_v$ and write it as a convex combination
\[ x = \sum_{i=1}^{k} \alpha_i v_i, \]
with $0 \leq \alpha_i \leq 1$ and $\sum_{i=1}^{k} \alpha_i =1$. Then
\[ x = \sum_{i=1}^{k} \alpha_i w_i + \sum_{i=1}^{k} \alpha_i (v_i - w_i). \]
If we set $y = \sum_{i=1}^{k} \alpha_i w_i$ and $z = \sum_{i=1}^{k} \alpha_i (v_i - w_i)$, then $y \in Q_v$. Moreover, by the triangle inequality $z \in B$, since each term $v_i - w_i$ lies in $B$. This shows that $P_v \subseteq Q_v + B$. 

If we denote the Minkowski difference by $\div$, then we have
\[ P_v \div B \subset (Q_v+B) \div B = Q_v,  \]
where the last equality holds since $Q_v$ and $B$ are non-empty, compact, and convex. See Lemma \ref{minkowski-difference}. Therefore, it suffices to show that $P_v^b \subset P_v \div B$.
It follows from the definition that $P_v^b+B \subset P_v$. Hence, assuming that $P_v^b$ is non-empty, we have
\[ P_v^b = (P_v^b+B) \div B \subset P_v \div B,  \]
where we used Lemma \ref{minkowski-difference} twice. This completes the proof.
\end{proof}

Define 
\begin{eqnarray*}
P_j^b := \{ x \in P_j \hspace{2mm} | \hspace{2mm} \text{dist}(x, \partial P_j)\geq b \},
\end{eqnarray*}
and set
\begin{align*}
V := Lv+ \{ x \in E_1^c \hspace{2mm} | \hspace{2mm} \pi_j(x) \in P_j^b \hspace{3mm} \text{for every } j>1 \},
\end{align*}
\begin{align*}
W : = Lv+ \{ x \in E_1^c \hspace{1mm}| \hspace{1mm} |\pi_j(x)|_\alpha \leq R_j -b \hspace{2mm} \text{for } 1 < j \leq r, \hspace{1mm}   |\pi_j(x)|_\alpha \leq R_j (1 - \frac{5}{N_j^2} ) -b \hspace{2mm} \text{for } j>r \}.
\end{align*}

\begin{lem}
We have $W \subset V \subset P_v^b \subset Q_v$.
\label{inclusion-W-Q}
\end{lem}

\begin{proof}
The inclusion $W \subset V$ follows from the following simple fact: Let $T$ be a regular $N$-gon inscribed in a circle of radius $R$ centered at the point $O$. Every point in $\partial T$ is of distance at least $R(1 - 5/N^2)$ from $O$.

For completeness, we give a proof of the above fact. After scaling, we may assume that $R=1$. The minimum distance $d(O ,x)$ for $x \in \partial T$ is obtained by drawing the perpendicular from $O$ to a side of $T$. Such a perpendicular has length $\cos(\frac{\pi}{N})$. Hence
\[ d(O, x) \geq \cos(\frac{\pi}{N}) > 1 - \frac{\pi^2}{2N^2}> 1 - \frac{5}{N^2}, \]
where we have used the inequality $\cos(x) > 1 - \frac{x^2}{2}$ for $0 < x <1$.

The inclusion $P_v^b \subset Q_v$ was proved in Lemma \ref{inclusion-U}. Now we show that $V \subset P_v^b$. Pick arbitrary points $x \in V$ and $y \in \partial P_v = \partial (Lv + P)$. Since $P = \prod_{j>1} P_j$ is a product, there is $j>1$ such that $\pi_j(y) \in \partial P_j$. Therefore 
\[ \text{dist}(x, y) \geq \text{dist}(\pi_j(x), \pi_j(y)) \geq \text{dist}(P_j^b , \partial P_j) \geq b.  \]
In the above, we used the fact that projection onto a non-empty closed convex set in a Euclidean space is distance decreasing; see e.g. \cite[Theorem 1.2.1]{schneider_2013}. Hence $x \in P_v^b$, proving the inclusion $V \subset P_v^b$.
\end{proof}

Using (\ref{cone-C}), in order to show that $\mathcal{C}$ is invariant, it is enough to prove that for every $z_i$
\begin{align*}
L \cdot \frac{M_\lambda(z_i)}{|\pi_1(M_\lambda(z_i))|_\alpha} \in  W  \subset Q_v.
\end{align*}
Equivalently, we would like to show that for every $j>1$ and every $z_i$ the following inequalities hold
\begin{align*}
&\frac{|\pi_j (M_\lambda(z_i))|_\alpha}{|\pi_1(M_\lambda(z_i))|_\alpha} \leq  \frac{R_j -b}{L} \hspace{3mm} \text{for } 1 < j \leq r  , \\
&\frac{|\pi_j(M_\lambda(z_i))|_\alpha}{|\pi_1(M_\lambda(z_i))|_\alpha} \leq\frac{1}{L} \bigg( R_j \big(1 - \frac{5}{N_j^2} \big) - b \bigg) \hspace{3mm} \text{for } j>r.
\end{align*}
But the above inequalities hold by Lemma \ref{estimates}. This finishes the proof of the invariance of $\mathcal{C}$. \\

\textbf{Step 5: Defining a non-negative integral matrix $A$, where each irreducible component of $A$ has spectral radius equal to $\lambda$.}\\

Let $S$ be the semigroup generated by elements of $\mathcal{C} \cap \mathcal{O}_\mathbb{K}$ under vector addition. By Proposition \ref{semigroup}, the set of integral points (i.e. elements of $\mathcal{O}_\mathbb{K}$) in the compact region
\begin{eqnarray}
C : = \{ \alpha_1 z_1 + \alpha_2 z_2 + \cdots + \alpha_k z_k \hspace{1mm} | \hspace{1mm} 0 \leq \alpha_i \leq 1 \hspace{3mm} \text{for every } i\}
\label{compact-C} 
\end{eqnarray}
generate $S$. Enumerate the set of points in $C \cap \mathcal{O}_\mathbb{K}$ by $c_1, \cdots, c_n$. Recall that $M_\lambda$ acts on $S$. Moreover, $M_\lambda$ can be represented by the companion matrix $B$ when written in the basis $\{ 1, \lambda, \cdots, \lambda^{d-1} \}$ of $\mathbb{K}_\mathbb{R} = \mathbb{Q}(\lambda) \otimes \mathbb{R}$. Let $A=[a_{ij}]$ be a non-negative integral matrix corresponding to the action of $B$ on $S$ in the basis $\{ c_1, \cdots, c_n\}$. Hence 
\begin{eqnarray}
Bc_j = \sum_{i=1}^{n} a_{ij} c_i,
\label{matrix-A}
\end{eqnarray}
where $B \colon \mathbb{R}^d \rightarrow \mathbb{R}^d$ is the companion matrix associated with $\lambda$. There might be various choices for $A$, when $Bc_j$ can be written as a non-negative linear combination of $c_1, \cdots, c_n$ in more than one way.

Lind's argument \cite[pages 288--289]{lind1984entropies} shows that every irreducible component of $A$ has spectral radius $\lambda$. We follow the proof given in \cite[Lemma 11.1.10]{lind_marcus_2021} briefly for the reader's convenience. Replace $A$ by one of its irreducible components; this amounts to taking a minimal subset of $\{ c_1 , \cdots, c_n \}$ for which (\ref{matrix-A}) holds. Let $\mu$ be the spectral radius of $A$. Let $e_i$ be the $i$th unit vector in $\mathbb{R}^n$, and define the linear map $P \colon \mathbb{R}^n \rightarrow \mathbb{R}^d$ by $P(e_i)=c_i$. Then by (\ref{matrix-A}) we have $PA=BP$. Since $A$ is non-negative and irreducible, by the Perron--Frobenius theorem $A$ has a positive eigenvector $v_\mu$ corresponding to $\mu$. Since $Pv_\mu$ is a positive linear combination of the $c_j$, and each $c_j$ satisfies $\pi_1(c_j)>0$, we necessarily have $\pi_1(Pv_\mu)>0$ and hence $Pv_\mu \neq 0$. Moreover, 
\[ B(Pv_\mu)=P(Av_\mu)=\mu(Pv_\mu). \]
Hence, $Pv_\mu$ is an eigenvector for $B$ with eigenvalue $\mu$ . But the only eigenvectors of $B$ with positive $E_1$ component lie in $E_1$. Hence $Pv_\mu$ is a multiple of $v$ (the eigenvector with eigenvalue $\lambda$) and 
\begin{align*}
 \mu (Pv_\mu) &= B(Pv_\mu)=\lambda (Pv_\mu) \\
 &\implies \mu = \lambda.
\end{align*}

\begin{remark}
In \cite[page 354]{thurston2014entropy}, Thurston gave a new proof of Lind's converse to the integer Perron--Frobenius theorem. Thurston assumed that the matrix constructed via his method can be taken to be primitive as well. If this is the case, then Theorem \ref{thm:upper-bound} readily upgrades to give an upper bound for the Perron--Frobenius degree of a Perron number. 
\label{primitive-case}
\end{remark}

\textbf{Step 6: Bounding the dimension of the matrix $A$.}\\ 

By Proposition \ref{number of lattice points}, the number of integral points in $C$ is at most 
\begin{align*}
\frac{\textsc{Vol}(C)}{\text{Covol}(\mathcal{O}_{\mathbb{K}})} \cdot (d+1)! = \frac{\textsc{Vol}(C)}{\text{det}(\mathcal{O}_\mathbb{K} , q_\alpha)^\frac{1}{2}} \cdot (d+1)!.
\end{align*}
We give an upper bound for $\textsc{Vol}(C)$ via a series of lemmas. 

\begin{lem}
For every $j>1$ and every $w_i$, we have
\[ |\pi_j(w_i)|_\alpha < 2 R_j. \]
\label{length-of-projection-w}
\end{lem}

\begin{proof}
By the triangle inequality
\begin{align*}
|\pi_j(w_i)|_\alpha &\leq |\pi_j(v_i)|_\alpha + |\pi_j(w_i - v_i)|_\alpha \\ 
& \leq |\pi_j(v_i)|_\alpha + |w_i - v_i|_\alpha \\ 
&\leq R_j + (2 \sqrt{d}+2)\ell < 2 R_j \\
& \iff R_j > (2 \sqrt{d}+2)\ell.
\end{align*}
In the first line above (\ref{norm-comparison}) is used. The second line follows from Lemma \ref{upper-bound-v}, and the last inequality holds by (\ref{R-real}), (\ref{lower-bound-N}), and (\ref{R-complex}).
\end{proof}

\begin{lem}
Let $M:= 2 k \sqrt{d}$, where $k$ is the number of vertices of $P_v$. Then
\[ C \subset \{ rx \in \mathbb{R}^d \hspace{2mm} | \hspace{2mm} x \in Q_v, \hspace{1mm} \text{and} \hspace{2mm} 0 \leq r \leq M \}.  \]
\label{inclusion-C}
\end{lem}

\begin{proof}
Since $Q_v$ is the convex hull of $w_i$, the ray from the origin and passing thorough an arbitrary point $p$ of $C$ intersects $Q_v$ at some point $w$. Therefore, if we define $M_1$ as 
\[ M_1 = \frac{\max_{p \in C} |p|_\alpha}{\min_{w \in Q_v} |w|_\alpha}, \]
then 
\[ C \subset \{ rx \in \mathbb{R}^d \hspace{2mm} | \hspace{2mm} x \in Q_v, \hspace{1mm} \text{and} \hspace{2mm} 0 \leq r \leq M_1 \}.  \]
It is enough to show that $M_1 \leq M$. For every point $w$ in $Q_v$ 
\[ |w|_\alpha \geq |\pi_1(w)|_\alpha = L. \] 
On the other hand if $p$ is a point in $C$, then by the triangle inequality we have 
\[ |p|_\alpha \leq k \cdot \max_{i} |z_i|_\alpha \leq k \cdot \max_i(|v_i|_\alpha+\ell) \leq k (\sqrt{d}L + \ell) < 2k\sqrt{d}L, \]
where we used Lemma \ref{upper-bound-v} for the implication $|v_i|_\alpha \leq \sqrt{d} L$. Combining the previous inequalities gives
\[ M_1 = \frac{\max_{p \in C} |p|_\alpha}{\min_{w \in Q_v} |w|_\alpha} \leq \frac{2k \sqrt{d}L}{L} = 2 k \sqrt{d}= M. \] 
\end{proof}

\begin{lem}
	We have 
	\[ \textsc{Vol}(Q_v) \leq 2^{2d-2} \times \prod_{j>1}^{r} R_j \times \prod_{j >r}R_j^2, \]
	where $\textsc{Vol}(Q_v)$ is the $(d-1)$-dimensional volume of $Q_v$.
	\label{upper-bound-Qv}
\end{lem}

\begin{proof}
	We have
	\begin{align*}
	\textsc{Vol}(Q_v) \leq \prod_{j>1}^{r} \textsc{Length}(\pi_j(Q_v)) \times \prod_{j >r} \textsc{Area}(\pi_j(Q_v)), 
	\end{align*}
	where $\textsc{Vol}$, $\textsc{Length}$ and $\textsc{Area}$ are with respect to the inner product $q_\alpha$. For $j>1$, define $\textbf{R}_j$ as
	\[ \textbf{R}_j = \max_{w \in Q_v} |\pi_j(w)|_\alpha.  \] 
	Therefore 
	\[ \textsc{Vol}(Q_v) \leq \prod_{j>1}^{r} (2\textbf{R}_j) \times \prod_{j >r} (\pi \textbf{R}_j^2).  \]
	By Lemma \ref{length-of-projection-w} and the triangle inequality we have $\textbf{R}_j < 2 R_j$. The lemma now follows from the inequalities $\textbf{R}_j < 2 R_j$ and $\pi < 4$, and the equality $r+2s=d$.
\end{proof}

\begin{lem}
	We have 
	\[ \textsc{Vol}(C) < 2^{d^2+6d-1} \times d^{\frac{ds}{4}+\frac{3d}{2} -1} \times \frac{\ell^{d}}{(1-\rho)^{d+\frac{ds}{2}}}, \]
	where $C$ is defined as in (\ref{compact-C}).
\end{lem}

\begin{proof}
	Note that
	\begin{eqnarray}
	M = 2 k \sqrt{d} = 2^r \times \prod_{j >r} N_j \times \sqrt{d} \leq 2^r \times \Big(\frac{16 \sqrt{d}}{1- \rho}\Big)^{\frac{s}{2}} \times \sqrt{d} < 2^d \times \frac{d^{\frac{s}{4}+1}}{(1-\rho)^\frac{s}{2}}, 
	\end{eqnarray}
	where we used (\ref{number-of-vertices-k}), (\ref{upper-bound-R-complex}), and the identity $r+2s=d$. By Lemma \ref{inclusion-C}
	\begin{align*}
	\textsc{Vol}(C) &\leq \textsc{Vol}(\{ rx \in \mathbb{R}^d \hspace{1mm} | \hspace{1mm} x \in Q_v, \hspace{1mm} \text{and} \hspace{2mm} 0 \leq r \leq M \} ).
	\end{align*}
	The upper bound above is the volume of a cone over $M \cdot Q_v$, where $M \cdot Q_v$ denotes a dilation of $Q_v$ with the factor of $M$ from the point $0 \in \mathbb{R}^d$. The height of this cone is $ML$, since the height of $Q_v$ measured perpendicularly from its apex along the $E_1$-axis is $L$. Hence, the volume of the cone is equal to 
	\[ \frac{(ML)\textsc{Vol}(M \cdot Q_v)}{d} = \frac{M^d L }{d} \textsc{Vol}(Q_v).  \]
	Therefore, we have 
	
	\begin{align*}
	\textsc{Vol}(C) \leq  \frac{M^d L}{d} \times \textsc{Vol}(Q_v) & \leq \Big(2^d \times \frac{d^{\frac{s}{4}+1}}{(1-\rho)^\frac{s}{2}} \Big)^d \times \frac{L}{d} \times 2^{2d-2} \times \prod_{j>1}^{r} R_j \times \prod_{j >r}R_j^2,
	\end{align*}
	where we used Lemma \ref{upper-bound-Qv}. Moreover, using (\ref{upper-bound-R-real}), (\ref{upper-bound-R-complex}), and (\ref{R-complex}) we have
	\begin{align*}
	\prod_{j>1}^{r} R_j \times \prod_{j >r}R_j^2 &\leq \prod_{j>1}^{r} \Big(\frac{8\sqrt{d} \ell }{1-\rho}\Big) \times \prod_{j >r} \Big(\frac{16 \sqrt{d} \ell }{1 - \rho}\Big)^2 \\
	& = 2^{3r +8s - 3} \times d^{\frac{d-1}{2}}  \times \frac{\ell^{d-1}}{(1-\rho)^{d-1}}.
	\end{align*}
	Combining with the previous upper bound for $\textsc{Vol}(C)$, and using $L < \frac{16 \sqrt{d} \ell}{1-\rho}$, we have
	\begin{eqnarray*}
		\textsc{Vol}(C) \leq 2^{d^2 +2d +3r + 8s -1} \times d^{\frac{ds}{4}+\frac{3d}{2} -1} \times \frac{\ell^{d}}{(1-\rho)^{d+\frac{ds}{2}}} \leq 2^{d^2+6d-1} \times d^{\frac{ds}{4}+\frac{3d}{2} -1} \times \frac{\ell^{d}}{(1-\rho)^{d+\frac{ds}{2}}},
	\end{eqnarray*}
	where we used the relation $r+2s =d$. 
	
\end{proof}

Therefore, the number of integral points in $C$ is at most 
\begin{align*}
\frac{\textsc{Vol}(C)}{\text{det}(\mathcal{O}_\mathbb{K} , q_\alpha)^\frac{1}{2}} \cdot (d+1)! &\leq 2^{d^2+6d} \times d^{\frac{ds}{4}+\frac{5d}{2} -1} \times \frac{\ell^d}{\text{det}(\mathcal{O}_\mathbb{K} , q_\alpha)^\frac{1}{2}}  \times \frac{1}{(1-\rho)^{d + \frac{ds}{2}}} \\
& = 2^{d^2+6d} \times d^{\frac{ds}{4}+\frac{5d}{2} -1} \times \tau(\mathcal{O}_\mathbb{K}, q_\alpha)^\frac{d}{2} \times \frac{1}{(1-\rho)^{d + \frac{ds}{2}}}.
\end{align*}
In the first line above, we used the inequality 
\[ (d+1)! = (d+1) d! < (2d) d^{d-1}= 2d^d. \] 
The second line above uses the definition of $\ell$, and Definition \ref{def:tau-min}. By taking infimum over all $\alpha \in \mathfrak{B}$, we obtain the upper bound
\begin{eqnarray*}
2^{d^2+6d} \times d^{\frac{ds}{4}+\frac{5d}{2} -1} \times \frac{\tau_{\min}(\mathcal{O}_\mathbb{K})^\frac{d}{2}}{ (1-\rho)^{d+\frac{ds}{2}}},
\end{eqnarray*}
for the number of integral points in $C$ and hence for the dimension of the matrix $A$. 

Bayer Fluckiger showed in \cite[Propositon 4.2]{BAYERFLUCKIGER2006305}, as a corollary of the work of Banaszczyk \cite[Theorem 2.2]{banaszczyk1993new}, that for any $\alpha \in \mathfrak{B}$
\begin{eqnarray}
\tau(\mathcal{O}_\mathbb{K}, q_\alpha) \leq \frac{d}{4} \cdot D_\mathbb{K}^\frac{1}{d}.
\label{upper-bound-tau}
\end{eqnarray}
This gives the upper bound 
\[  2^{d^2+5d} \times d^{\frac{ds}{4}+3d -1} \times \frac{\sqrt{D_\mathbb{K}}}{ (1-\rho)^{d+\frac{ds}{2}}} \]
for the dimension of $A$. Both parts of the theorem now follow from the crude estimates
\[ 2^{d^2+6d} \leq 8^{d^2}, \hspace{6mm} \frac{ds}{4}+3d -1 <  d^2, \hspace{6mm} d+ \frac{ds}{2}< d^2. \]

\end{proof}

The bound in Theorem \ref{thm:upper-bound} is perhaps enormous compared to the Perron--Frobenius degree, so we have not tried to make the constants optimal. The point is to have an explicit bound in terms of data that we believe are relevant to the Perron--Frobenius degree; see Question \ref{arithmetic-information}.

\begin{remark}
 
Denote the \emph{covering conjecture} for dimension $d$ by $C_d$.

\begin{conj}[Covering conjecture]
	The covering radius of any well-rounded unimodular lattice $L$ in $\mathbb{R}^d$ (with the standard norm $|\cdot|$) satisfies
	\[ \sup_{x \in \mathbb{R}^d} \inf_{y \in L} |x -y| \leq \frac{\sqrt{d}}{2}.  \]
	Equality happens if and only if $L = g \cdot \mathbb{Z}^d$ for some $g \in \mathrm{SO}_d(\mathbb{R})$. 
\end{conj}

See McMullen \cite{mcmullen2005minkowski} for the definition of \emph{well-rounded lattice}, and the application of the covering conjecture to \emph{Minkowski's conjecture}. The covering conjecture is proved for $d \leq 10$; see \cite{kathuria2020conjectures,kathuria2016conjectures} and the references therein. Moreover, it is known to be false for $d \geq 30$; see \cite{regev2017}.
	
In \cite[page 313]{BAYERFLUCKIGER2006305}, Bayer Fluckiger pointed out that McMullen's results \cite{mcmullen2005minkowski} together with the covering conjecture $C_d$ imply that for any totally real $\lambda$ of degree $d$ we have $\tau_{\min} (\mathcal{O}_\mathbb{K}) \leq \frac{d}{4}$. Hence, for any $d \leq 10$ and for any totally real Perron number $\lambda$ of degree $d$
\[ d_{PF}^{irr}(\lambda) \leq \Big( \frac{8 d}{1 - \rho} \Big)^{d^2}.  \]

It would be interesting to know which number fields satisfy an inequality similar to that of $\tau_{\min} (\mathcal{O}_\mathbb{K}) \leq \frac{d}{4}$ with upper bound only depending on the degree $d$. As pointed out kindly by McMullen to me, it is easy to see that such a bound does not exist for imaginary quadratic fields $\mathbb{Q}(\sqrt{n})$ for $n<0$. On the other hand, Bayer Fluckiger showed that the inequality $\tau_{\min} (\mathcal{O}_\mathbb{K}) \leq \frac{d}{4}$ holds for fields of the form $\mathbb{K} = \mathbb{Q}(\zeta_{p^r})$ or $\mathbb{K} = \mathbb{Q}(\zeta_{p^r}+\zeta^{-1}_{p^r})$, where $p$ is an odd prime number, $r$ is a natural number, and $\zeta_{p^r}$ is a primitive $p^r$th root of unity; see respectively \cite[page 319, line 4]{BAYERFLUCKIGER2006305} and \cite[Lemma 8.5]{BAYERFLUCKIGER2006305}.

\label{rem:Minkowski}
\end{remark}


\begin{ex}[Pisot numbers]
A real algebraic integer $\alpha>1$ is called \emph{Pisot} if all other Galois conjugates of $\alpha$ lie in the unit circle $\{ z \in \mathbb{C} \hspace{2mm} | \hspace{2mm} |z|<1\}$. Note a Pisot number is always Perron. We collect a few facts about Pisot numbers. 

\begin{enumerate}
	\item A number field is called \emph{real} if it has at least one real place. It is clear that every number field containing a Pisot number with the same degree should be real. Pisot \cite{pisot1938repartition} proved that every real number field $\mathbb{K}$ of degree $d$ contains a Pisot number of degree $d$. Moreover, the set of Pisot numbers of degree $d$ in $\mathbb{K}$ is closed under multiplication; see \cite[Corollary in page 33]{meyer1972algebraic}. 
	\item The smallest Pisot number is the largest root $p$ of $x^3 - x -1$ and is known as the \emph{plastic constant}. This was identified as the smallest known Pisot number by Salem \cite{salem1944remarkable}, and Siegel proved it to be the smallest possible Pisot number \cite{siegel1944algebraic}. In particular, if $\lambda$ is a Pisot number with spectral ratio $\rho$, then $\rho < p^{-1}$ implying that
	\[  \frac{1}{1-\rho} < \frac{1}{1 - p^{-1}}. \] 
\end{enumerate} 
	
Let $\mathbb{K}$ be a real number field of degree $d \geq 3$. Let $\lambda$ be any Pisot number in $\mathbb{K}$ of degree $d$, and note that $\mathbb{Q}(\lambda) = \mathbb{K}$. By Theorem \ref{thm:upper-bound}, $d_{PF}^{irr}(\lambda)$ is bounded above by 
\[ \Big(\dfrac{8d}{1-p^{-1}}\Big)^{d^2} \sqrt{D_\mathbb{K}}. \]
Note this upper bound only depends on $\mathbb{K}$ and not on the Pisot number $\lambda \in \mathbb{K}$. Every number field $\mathbb{K}$ has only finitely many subfields $\mathbb{K}' \subset \mathbb{K}$. Hence there is a similar upper bound depending only on $\mathbb{K}$ for an arbitrary Pisot number $\lambda \in \mathbb{K}$ (with any degree).
\label{pisot}
\end{ex}

\section{Primitive matrices}
\label{sec:primitive}
In this section, we use Theorem \ref{thm:upper-bound} to derive an upper bound for the Perron--Frobenius degree of a Perron number $\lambda$. A useful observation is that if there exists an irreducible matrix $A$ with spectral radius $\lambda -1$, then $\lambda$ is the spectral radius of the \emph{primitive} matrix $I+A$. This is because $I+A$ is irreducible and has positive trace; hence it is primitive. 

\newtheorem*{thm:primitive}{Theorem \ref{thm:primitive}}
\begin{thm:primitive}
Let $\mathbb{\lambda}$ be a Perron number of degree $d \geq 3$ and spectral ratio $\rho$. Set $\mathbb{K}:= \mathbb{Q}(\lambda)$. Denote the bound from Theorem \ref{thm:upper-bound} by $B(\mathbb{K}, \rho)$, and let $\kappa(\mathbb{K}, \rho)$ be as in Notation \ref{notation:kappa}. The Perron--Frobenius degree of $\lambda$ is bounded above by 
\[ \max \{ 2^{d^2} B(\mathbb{K}, \rho), \kappa(\mathbb{K},\rho) \}. \]

\end{thm:primitive}

\begin{proof}
We may assume that $\lambda \geq M = 1+ \frac{4}{1-\rho}$. First we show that $\lambda -1$ is Perron. The Galois conjugates of $\lambda -1$ are equal to $\lambda_i -1$. We have
\begin{align*}
\frac{|\lambda_i -1|}{\lambda -1} &\leq \frac{|\lambda_i|+1}{\lambda -1} \leq \frac{\rho \lambda +1}{\lambda -1} = \rho + \frac{\rho +1}{\lambda -1} \leq \\ 
&\rho + \frac{2}{\lambda -1} \leq \rho + \frac{1-\rho}{2} <1,
\end{align*}
where we used the assumptions $\rho <1$ and $\lambda \geq 1+ \frac{4}{1-\rho}$. Denote the spectral ratio for $\lambda -1$ by $\mu$. Note that 
\begin{align}
1- \mu \geq \frac{1 - \rho}{2}. 
\label{mu}
\end{align}
In fact, we just showed that for every $i$
\begin{align*}
\frac{|\lambda_i -1|}{\lambda -1} \leq \rho + \frac{1-\rho}{2},
\end{align*}
hence

\begin{align*}
1 - \frac{|\lambda_i -1|}{\lambda -1} \geq 1 - (\rho + \frac{1- \rho}{2}) = \frac{1- \rho}{2}.
\end{align*}
Inequality (\ref{mu}) in particular implies that
\begin{align}
B(\mathbb{K}, \mu) \leq 2^{d^2} B(\mathbb{K}, \rho).
\end{align}
Now $\lambda -1$ is a Perron number of spectral ratio $\mu$ and it generates the same number field $\mathbb{K}$ as $\lambda$. By Theorem \ref{thm:upper-bound}, there is an integral irreducible matrix $A$ with spectral radius $\lambda -1$ and dimension at most $B(\mathbb{K}, \mu)$. Then $I + A$ is a primitive matrix of the same dimension and with spectral radius $\lambda$. 

\end{proof}

\section{Questions}
\label{sec:questions}
In Theorem \ref{thm:upper-bound}, for technical reasons, we constructed an integral \emph{irreducible} matrix with spectral radius equal to a given Perron number, although it would have been more natural to construct an integral \emph{primitive} matrix instead. This motivates the following:

\begin{question} 
\begin{enumerate}
\item Are there upper bounds for $d_{PF}(\lambda)$ in terms of $d_{PF}^{irr}(\lambda)$?
\item Does $d_{PF}(\lambda) = d_{PF}^{irr}(\lambda)$ always hold?
\end{enumerate}	

\label{ques:primitive-vs-irreducible}
\end{question}

The lower bound for the Perron--Frobenius degree in \cite{yazdi2020lower} is in terms of the two largest Galois conjugates in the complex plane. Therefore, the following is a natural question. 
\begin{question}
Let $d$ and $\rho$ denote, respectively, the algebraic degree and the spectral ratio. Is there an upper bound $B=B(d,\rho)$ for $d_{PF}^{irr}(\lambda)$ (respectively $d_{PF}(\lambda)$) where $\lambda$ is an arbitrary Perron number? 
\label{arithmetic-information}
\end{question}

We expect the above question to have a negative answer, meaning that in Theorem \ref{thm:upper-bound} the arithmetic information on $\mathcal{O}_\mathbb{K}$ cannot be overlooked. 
A unit algebraic integer $\lambda$ is \emph{bi-Perron} if all other Galois conjugates of $\lambda$ lie in the annulus $\{ c \in \mathbb{C} \hspace{1mm} | \hspace{1mm} \lambda^{-1} < |z| < \lambda \}$ except possibly for $\pm \lambda^{-1}$. See \cite{McMullen}.


\begin{question}
How can the bound in Theorem \ref{thm:upper-bound} be improved for special classes of algebraic integers such as totally real Perron numbers, Pisot numbers, Salem numbers, or bi-Perron numbers? 
\end{question}

In particular, we ask the following about Pisot numbers.

\begin{question}
Let $d$ denote the algebraic degree. Is there an upper bound $B = B(d)$ for $d_{PF}^{irr}(\lambda)$ (respectively $d_{PF}(\lambda)$) where $\lambda$ is an arbitrary Pisot number?
\end{question}

\bibliographystyle{alpha}
\bibliography{references-Perron}

\begin{thebibliography}{{Bay}06}

\bibitem[Ban93]{banaszczyk1993new}
Wojciech Banaszczyk.
\newblock New bounds in some transference theorems in the geometry of numbers.
\newblock {\em Mathematische Annalen}, 296(1):625--635, 1993.

\bibitem[{Bay}06]{BAYERFLUCKIGER2006305}
Eva {Bayer Fluckiger}.
\newblock Upper bounds for {E}uclidean minima of algebraic number fields.
\newblock {\em Journal of Number Theory}, 121(2):305 -- 323, 2006.

\bibitem[BH91]{boyle1991spectra}
Mike Boyle and David Handelman.
\newblock The spectra of nonnegative matrices via symbolic dynamics.
\newblock {\em Annals of Mathematics}, pages 249--316, 1991.

\bibitem[BL02]{boyle2002small}
Mike Boyle and Douglas Lind.
\newblock Small polynomial matrix presentations of non-negative matrices.
\newblock {\em Linear algebra and its applications}, 355:49--70, 2002.

\bibitem[Jar14]{jarvis2014algebraic}
Frazer Jarvis.
\newblock {\em Algebraic number theory}.
\newblock Springer, 2014.

\bibitem[KOR00]{kim2000spectra}
Ki~Kim, Nicholas Ormes, and Fred Roush.
\newblock The spectra of nonnegative integer matrices via formal power series.
\newblock {\em Journal of the American Mathematical Society}, 13(4):773--806,
  2000.

\bibitem[KR16]{kathuria2016conjectures}
Leetika Kathuria and Madhu Raka.
\newblock On conjectures of {M}inkowski and {W}oods for n= 9.
\newblock {\em Proceedings-Mathematical Sciences}, 126(4):501--548, 2016.

\bibitem[KR20]{kathuria2020conjectures}
Leetika Kathuria and Madhu Raka.
\newblock On conjectures of {M}inkowski and {W}oods for $ n= 10$.
\newblock {\em arXiv preprint arXiv:2009.09992}, 2020.

\bibitem[Lin84]{lind1984entropies}
Douglas~A Lind.
\newblock The entropies of topological {M}arkov shifts and a related class of
  algebraic integers.
\newblock {\em Ergodic Theory and Dynamical Systems}, 4(2):283--300, 1984.

\bibitem[LM21]{lind_marcus_2021}
Douglas Lind and Brian Marcus.
\newblock {\em An Introduction to Symbolic Dynamics and Coding}.
\newblock Cambridge Mathematical Library. Cambridge University Press, 2
  edition, 2021.

\bibitem[McM05]{mcmullen2005minkowski}
Curtis McMullen.
\newblock Minkowski’s conjecture, well-rounded lattices and topological
  dimension.
\newblock {\em Journal of the American Mathematical Society}, 18(3):711--734,
  2005.

\bibitem[McM14]{McMullen}
Curtis McMullen.
\newblock Slides for {D}ynamics and algebraic integers: Perspectives on
  {T}hurston’s last theorem.
\newblock {\em
  http://www.math.harvard.edu/~ctm/expositions/home/text/talks/cornell/2014/slides/slides.pdf},
  2014.

\bibitem[Mey72]{meyer1972algebraic}
Yves Meyer.
\newblock {\em Algebraic Numbers and Harmonic Analysis}, volume~2.
\newblock Elsevier, 1972.

\bibitem[Pis38]{pisot1938repartition}
Charles Pisot.
\newblock La r{\'e}partition modulo 1 et les nombres alg{\'e}briques.
\newblock {\em Annali della Scuola Normale Superiore di Pisa-Classe di
  Scienze}, 7(3-4):205--248, 1938.

\bibitem[RSW17]{regev2017}
Oded Regev, Uri Shapira, and Barak Weiss.
\newblock Counterexamples to a conjecture of {W}oods.
\newblock {\em Duke Math. J.}, 166(13):2443--2446, 09 2017.

\bibitem[Sal44]{salem1944remarkable}
Raphael Salem.
\newblock A remarkable class of algebraic integers. {P}roof of a conjecture of
  {V}ijayaraghavan.
\newblock {\em Duke Mathematical Journal}, 11(1):103--108, 1944.

\bibitem[Sch13]{schneider_2013}
Rolf Schneider.
\newblock {\em Convex Bodies: The Brunn–Minkowski Theory}.
\newblock Encyclopedia of Mathematics and its Applications. Cambridge
  University Press, 2 edition, 2013.

\bibitem[Sie44]{siegel1944algebraic}
Carl~Ludwig Siegel.
\newblock Algebraic integers whose conjugates lie in the unit circle.
\newblock {\em Duke Mathematical Journal}, 11(3):597--602, 1944.

\bibitem[Thu89]{thurston1989groups}
William~P Thurston.
\newblock Groups, tilings and finite state automata.
\newblock In {\em AMS Colloq. Lectures}, 1989.

\bibitem[Thu14]{thurston2014entropy}
William Thurston.
\newblock Entropy in dimension one.
\newblock {\em Frontiers in Complex Dynamics: In Celebration of John Milnor's
  80th Birthday}, 2014.

\bibitem[Yaz21]{yazdi2020lower}
Mehdi Yazdi.
\newblock Lower bound for the {P}erron–{F}robenius degrees of {P}erron
  numbers.
\newblock {\em Ergodic Theory and Dynamical Systems}, 41(4):1264--1280, 2021.

\end{thebibliography}

\end{document}